\documentclass[12pt,a4paper]{amsart}
\usepackage{amsfonts,amsmath,amssymb}
\usepackage{hyperref}
\usepackage{latexsym,fullpage,amsfonts,amssymb,amsmath,amscd,graphics,epic}
\usepackage[all]{xy}
\usepackage{amssymb,amsthm,amsxtra}
\usepackage[usenames]{color}
\usepackage{amscd}
\usepackage{amsthm}
\usepackage{amsfonts}
\usepackage{amssymb}
\usepackage{mathrsfs}
\usepackage{url}
\usepackage{bbm}
\usepackage{wasysym}
\usepackage{enumitem}

\theoremstyle{plain}
\newtheorem*{theorem*}{Theorem}
\newtheorem*{remark*}{Remark}
\newtheorem*{example*}{Example}
\newtheorem{lemma}{Lemma}[subsection]
\newtheorem{proposition}[lemma]{Proposition}
\newtheorem{corollary}[lemma]{Corollary}

\newtheorem*{conjecture*}{Conjecture}

\theoremstyle{definition}
\newtheorem{definition}[lemma]{Definition}
\newtheorem{cond}[lemma]{Condition}

\theoremstyle{remark}
\newtheorem{remark}[lemma]{Remark}

\newtheorem{notation}[lemma]{Notation}

\newcommand{\Hom}{\operatorname{Hom}}

\newcommand{\id}{\operatorname{Id}}

\newcommand{\Aut}{{\operatorname{Aut}}}

\newcommand{\End}{\operatorname{End}}

\newcommand{\bC}{{\mathbb C}}
\newcommand{\bZ}{{\mathbb Z}}

\newcommand{\lam}{{\lambda}}

\newcommand{\gl}{{\mathfrak{gl}}}

\newcommand{\abs}[1]{\left|{#1}\right|}

\newcommand{\InnaA}[1]{#1}

\newcommand{\InnaC}[1]{{{{#1}}}}
\newcommand{\InnaD}[1]{{#1}}
\def\quotient#1#2{%
    \raise1ex\hbox{$#1$}\Big/\lower1ex\hbox{$#2$}%
}

\begin{document}

\date{\today}
\title{Notes on restricted inverse limits of categories}
 \author{Inna Entova Aizenbud}
\address{Inna Entova Aizenbud,
Massachusetts Institute of Technology,
Department of Mathematics,
Cambridge, MA 02139 USA.}
\email{inna.entova@gmail.com}

\begin{abstract}
We describe the framework for the notion of a restricted inverse limit of categories, with the main motivating example being the category of polynomial representations of the group $GL_\infty = \bigcup_{n \geq 0} GL_n$. This category is also known as the category of strict polynomial functors of finite degree, and it is the restricted inverse limit of the categories of polynomial representations of $GL_n$, $n \geq 0$. This note is meant to serve as a reference for future work.

\end{abstract}

\maketitle
\setcounter{tocdepth}{3}
\section{Introduction}
In this note, we discuss the notion of an inverse limit of an inverse sequence of categories and functors.

Given a system of categories $\mathcal{C}_i$ (with $i$ running through the set $\bZ_+$) and functors $\mathcal{F}_{i-1, i}: \mathcal{C}_i \rightarrow \mathcal{C}_{i-1}$ for each $i \geq 1$, we define the inverse limit category $\varprojlim_{i \in \bZ_+} \mathcal{C}_i$ to be the following category:
\begin{itemize}
 \item The objects are pairs $(\{C_i\}_{i \in \bZ_+}, \{\phi_{i-1, i}\}_{i \geq 1})$ where $C_i \in \mathcal{C}_i$ for each $i \in \bZ_+$ and $\phi_{i-1, i}: \mathcal{F}_{i-1, i}(C_i) \stackrel{\sim}{\rightarrow} C_{i-1}$ for any $i \geq 1$.
\item A morphism $f$ between two objects $(\{C_i\}_{i \in \bZ_+}, \{\phi_{i-1, i}\}_{i \geq 1})$, $(\{D_i\}_{i \in \bZ_+}, \{\psi_{i-1, i}\}_{i \geq 1})$ is a set of arrows $\{f_i: C_i \rightarrow D_i\}_{i \in \bZ_+}$ satisfying some compatability conditions.

\end{itemize}

This category is an inverse limit of the system $((\mathcal{C}_{i})_{i \in \bZ_+}, (\mathcal{F}_{i-1, i})_{i \geq 1})$ in the $(2, 1)$-category of categories with functors and natural isomorphisms. It is easily seen (see Section \ref{sec:inv_limit_cat}) that if the original categories $\mathcal{C}_{i}$ were pre-additive (resp. additive, abelian), and the functors $\mathcal{F}_{i-1, i}$ were linear (resp. additive, exact), then the inverse limit is again pre-additive (resp. additive, abelian).

One can also show that if the original categories $\mathcal{C}_{i}$ were monoidal (resp. symmetric monoidal, rigid symmetric monoidal) categories, and the functors $\mathcal{F}_{i-1, i}$ were, monoidal (resp. symmetric monoidal functors), then the inverse limit is again a monoidal (resp. symmetric monoidal, rigid symmetric monoidal) category.

\subsection{Motivating example: rings}\label{ssec:intro_motiv_ex_rings}
We now consider the motivating example.

First of all, consider the inverse system of rings of symmetric polynomials 
$$  ... \rightarrow \bZ[x_1, ..., x_n]^{S_n} \rightarrow \bZ[x_1, ..., x_{n-1}]^{S_{n-1}} \rightarrow ... \rightarrow \bZ[x_1] \rightarrow \bZ$$
with the homomorphisms given by $p(x_1, ..., x_n) \mapsto p(x_1, ..., x_{n-1}, 0)$.

We also consider the ring $\Lambda_{\bZ}$ of symmetric functions in infinitely many variables. This ring is defined as follows: first, consider the ring $\bZ[x_1, x_2, ...]^{\cup_{n \geq 0} S_n}$ of all power series with integer coefficients in infinitely many indeterminates $x_1, x_2, ...$ which are invariant under any permutation of indeterminates. The ring $\Lambda_{\bZ}$ is defined to be the subring of all the power series  such that the degrees of all its monomials are bounded.

We would like to describe the ring $\Lambda_{\bZ}$ as an inverse limit of the former inverse system.  
\begin{itemize}
 \item[{\bf $1$-st approach:}]

 The following construction is described in \cite[Chapter I]{Mac}. Take the inverse limit $\varprojlim_{n \geq 0} \bZ[x_1, ..., x_n]^{S_n}$ (this is, of course, a ring, isomorphic to $\bZ[x_1, x_2, ...]^{\cup_{n \geq 0} S_n}$), and consider only those elements $(p_n)_{n \geq 0}$ for which $deg(p_n)$ is a bounded sequence. These elements form a subring of $\varprojlim_{n \geq 0} \bZ[x_1, ..., x_n]^{S_n}$ which is isomorphic to the ring of symmetric functions in infinitely many variables.

  \item[{\bf $2$-nd approach:}]

  Note that the notion of degree gives a $\bZ_+$-grading on each ring $\bZ[x_1, ..., x_n]^{S_n}$, and on the ring $\Lambda_{\bZ}$. The morphisms $\bZ[x_1, ..., x_n]^{S_n} \rightarrow \bZ[x_1, ..., x_{n-1}]^{S_{n-1}}$ respect this grading; furthermore, they do not send to zero any polynomial of degree $n-1$ or less, so they define an isomorphism between the $i$-th grades of $\bZ[x_1, ..., x_n]^{S_n}$ and $\bZ[x_1, ..., x_{n-1}]^{S_{n-1}}$ for any $i < n$. One can then see that $\Lambda_{\bZ}$ is an inverse limit of the rings $\bZ[x_1, ..., x_n]^{S_n}$ in the category of $\bZ_+$-graded rings, and its $n$-th grade is isomorphic to the $n$-th grade of $\bZ[x_1, ..., x_n]^{S_n}$.
\end{itemize}

\subsection{Motivating example: categories}
We now move on to the categorical version of the same result.

Let $GL_n(\bC)$ (denoted by $GL_n$ for short) be the general linear group over $\bC$. We have an inclusion $GL_n \subset GL_{n+1}$ with the matrix $A \in GL_n$ corresponding to a block matrix $A' \in GL_{n+1}$ which has $A$ as the upper left $n \times n$-block, and $1$ in the lower right corner (the rest of the entries are zero). One can consider a similar inclusion of Lie algebras $\gl_n \subset \gl_{n+1}$.

Next, we consider the polynomial representations of the algebraic group $GL_n$ (alternatively, the Lie algebra $\gl_n$): these are the representations $\rho: GL_n \rightarrow \Aut(V)$ which can be extended to an algebraic map $Mat_{n \times n} (\bC) \rightarrow \End(V)$. These representations are direct summands of finite sums of tensor powers of the tautological representation $\bC^n$ of $GL_n$.

The category of polynomial representations of $GL_n$, denoted by $Rep(\gl_n)_{poly}$, is a semisimple symmetric monoidal category, with simple objects indexed by integer partitions with at most $n$ parts. The Grothendieck ring of this category is isomorphic to $\bZ[x_1, ..., x_n]^{S_n}$.

We also have functors $$\InnaD{\mathfrak{Res}}_{n-1, n} = ( \cdot)^{E_{n, n}} : Rep(\gl_{n})_{poly} \rightarrow Rep(\gl_{n-1})_{poly}$$ On the Grothendieck rings, these functors induce the homomorphisms $$\bZ[x_1, ..., x_n]^{S_n} \rightarrow \bZ[x_1, ..., x_{n-1}]^{S_{n-1}} \; \; \; \; p(x_1, ..., x_n) \mapsto p(x_1, ..., x_{n-1}, 0)$$ discussed above.

Finally, we consider the infinite-dimensional group $GL_{\infty} = \bigcup_{ n \geq 0} GL_n$, and its Lie algebra $\gl_{\infty} = \bigcup_{n \geq 0} \gl_n$. 
The category of polynomial representations of this group (resp. Lie algebra) is denoted by $Rep(\gl_{\infty})_{poly}$, and it is the free Karoubian symmetric monoidal category generated by one object (the tautological representation $\bC^{\infty}$ of $GL_{\infty}$). It is also known that this category is equivalent to the category of strict polynomial functors of finite degree (c.f. \cite{HY}), it is semisimple, and its Grothendieck ring is isomorphic to the ring $\Lambda_{\bZ}$. 

The category $Rep(\gl_{\infty})_{poly}$ possesses symmetric monoidal functors $$\Gamma_n : Rep(\gl_{\infty})_{poly} \rightarrow Rep(\gl_n)_{poly}$$ with the tautological representation of $\gl_{\infty}$ being sent to tautological representation of $\gl_n$. These functors are compatible with the functors $\InnaD{\mathfrak{Res}}_{n-1, n}$ (i.e. $\Gamma_{n-1} \cong \InnaD{\mathfrak{Res}}_{n-1, n} \circ \Gamma_n$), and the functor $\Gamma_n$ induces the homomorphism $$\Lambda_{\bZ} \rightarrow \bZ[x_1, ..., x_{n}]^{S_{n}} \; \; \; \; p(x_1, ..., x_n, x_{n+1}, ...) \mapsto p(x_1, ..., x_{n}, 0, 0, ...)$$

This gives us a fully faithful functor $\Gamma_{\text{lim}}:Rep(\gl_{\infty})_{poly} \rightarrow \varprojlim_{n \geq 0} Rep(\gl_n)_{poly}$. 

Finding a description of the image of the functor $\Gamma_{\text{lim}}$ inspires the following two frameworks for ``special'' inverse limits, which turn out to be useful in other cases as well.

\subsection{Restricted inverse limit of categories}


To define the restricted inverse limit, we work with categories $\mathcal{C}_{i}$ which are finite-length categories; namely, abelian categories where each object has a (finite) Jordan-Holder filtration. We require that the functors $\mathcal{F}_{i-1, i}$ be ``shortening'': this means that these are exact functors such that given an object $C \in \mathcal{C}_i$, we have $$\ell_{\mathcal{C}_{i-1}} (\mathcal{F}_{i-1, i}(C)) \leq \ell_{\mathcal{C}_{i}}(C)$$

In that case, it makes sense to consider the full subcategory of $\varprojlim_{i \in \bZ_+} \mathcal{C}_i$ whose objects are of the form $(\{C_i\}_{i \in \bZ_+}, \{\phi_{i-1, i}\}_{i \geq 1})$, with $ \{ \ell_{\mathcal{C}_{n}}(C_n) \}_{n \geq 0}$ being a bounded sequence (the condition on the functors implies that this sequence is weakly increasing).

This subcategory will be called the ``restricted'' inverse limit of categories $\mathcal{C}_{i}$ and will be denoted by $\varprojlim_{i \in \bZ_+, \text{ restr}}  \mathcal{C}_{i}$. It is the inverse limit of the categories $\mathcal{C}_{i}$ in the $(2, 1)$-category of finite-length categories and shortening functors.

Considering the restricted inverse limit of the categories $Rep(\gl_n)_{poly}$, we obtain a functor $$ \Gamma_{\text{lim}}:Rep(\gl_{\infty})_{poly} \rightarrow \varprojlim_{n \geq 0, \text{ restr}} Rep(\gl_n)_{poly}$$ It is easy to see that $ \Gamma_{\text{lim}}$ is an equivalence. Note that in terms of Grothendieck rings, this construction corresponds to the first approach described in Subsection \ref{ssec:intro_motiv_ex_rings}.

\subsection{Inverse limit of categories with filtrations}

Another construction of the inverse limit is as follows: let $K$ be a filtered poset, and assume that our categories $\mathcal{C}_{i}$ have a $K$-filtration on objects; that is, we assume that for each $k \in K$, there is a full subcategory $Fil_k(\mathcal{C}_{i})$, and the functors $\mathcal{F}_{i-1, i}$ respect this filtration (note that if we consider abelian categories and exact functors, we should require that the subcategories be Serre subcategories). 

We can then define a full subcategory $\varprojlim_{i \in \bZ_+, K-filtr}  \mathcal{C}_{i}$ of $\varprojlim_{i \in \bZ_+} \mathcal{C}_i$ whose objects are of the form $(\{C_i\}_{i \in \bZ_+}, \{\phi_{i-1, i}\}_{i \geq 1})$ such that there exists $k \in K$ for which $C_i \in Fil_k(\mathcal{C}_{i})$ for any $i \geq 0$. 

The category $\varprojlim_{i \in \bZ_+, K-filtr}  \mathcal{C}_{i}$ is automatically a category with a $K$-filtration on objects. It is the  inverse limit of the categories $\mathcal{C}_{i}$ in the $(2, 1)$-category of categories with $K$-filtrations on objects, and functors respecting these filtrations.

\begin{remark}
 \InnaC{A more general way to describe this setting would be the following.
 
 Assume that for each $i$, the category $\mathcal{C}_i$ is a direct limit of a system $$\left( (\mathcal{C}_i^k)_{k \in \bZ_+}, \left( \mathcal{G}_i^{k-1, k}: \mathcal{C}_i^{k-1} \rightarrow \mathcal{C}_i^k \right) \right)$$ Furthermore, assume that the functors $\mathcal{F}_{i-1, i}$ induce functors $\mathcal{F}_{i-1, i}^k: \mathcal{C}_{i-1}^k \rightarrow \mathcal{C}_i^k$ for any $k \in \bZ_+$, and that the latter are compatible with the functors $\mathcal{G}_{i}^{k-1, k}$. One can then define the category 
$$ \varinjlim_{k \in K} \varprojlim_{i \in \bZ_+} \mathcal{C}_i^k$$ which will be the ``directed'' inverse limit of the system. When $\mathcal{C}_i^k := Fil_k(\mathcal{C}_{i})$ and $\mathcal{G}_i^{k-1, k}$ are inclusion functors, the directed inverse limit coincides with $\varprojlim_{i \in \bZ_+, K-filtr}  \mathcal{C}_{i}$. 

All the statements in this note concerning inverse limits of categories with filtrations can be translated to the language of directed inverse limits.
 }
\end{remark}

Considering appropriate $\bZ_+$-filtrations on the objects of the categories $Rep(\gl_n)_{poly}$, we obtain a functor $$ \Gamma_{\text{lim}}:Rep(\gl_{\infty})_{poly} \rightarrow \varprojlim_{n \geq 0, \bZ_+ -filtr} Rep(\gl_n)_{poly}$$ One can show that this is an equivalence. Note that in terms of Grothendieck rings, this construction corresponds to the second approach described in Subsection \ref{ssec:intro_motiv_ex_rings} (in fact, in this particular case one can use a grading instead of a filtration; however, this is not the case in \cite{EA}).

\mbox{}

These two ``special'' inverse limits may coincide, as it happens in the case of the categories $Rep(\gl_n)_{poly}$, and in \cite{EA}. We give a sufficient condition for this to happen. In such case, each approach has its own advantages. 

The restricted inverse limit approach does not involve defining additional structures on the categories, and shows that the constructed inverse limit category does not depend on the choice of filtration, as long as the filtration satisfies some relatively mild conditions. 

Yet the object-filtered inverse limit approach is sometimes more convenient to work with, as it happens in \cite{EA}.

\section{Conventions}

Let $\mathcal{C}$ be an abelian category, and $C$ be an object of $\mathcal{C}$. A {\it Jordan-Holder
filtration} for $C$ is a finite sequence of subobjects of $C$
$$0 = C_0 \subset C_1 \subset ... \subset C_n = C$$
such that each subquotient $C_{i+1} / C_i$ is simple.

 The Jordan-Holder filtration might not be unique, but the simple factors $C_{i+1} / C_i$ are unique (up to re-ordering and isomorphisms). Consider the multiset of the simple factors: each simple factor is considered as an isomorphism class of simple objects, and its multiplicity is the multiplicity of its isomorphism class in the Jordan-Holder filtration of $C$. This multiset is denoted by $JH(C)$, and its elements are called the {\it Jordan-Holder components} of $C$. 
 
 The {\it length} of the object $C$, denoted by $\ell_{\mathcal{C}}(C)$, is defined to be the size of the finite multiset $JH(C)$.

\begin{definition}
An abelian category $\mathcal{C}$ is called a {\it finite-length category} if every object admits a Jordan-Holder filtration.
 
\end{definition}

\section{Inverse limit of categories}\label{sec:inv_limit_cat}
In this section we discuss the notion of an inverse limit of categories, based on \cite[Definition 1]{WW}, \cite[Section 5]{Sch}.
This is the inverse limit in the $(2, 1)$-category of categories with functors and natural isomorphisms.

\subsection{Inverse limit of categories}\label{subsec:inv_limit_def}
Consider the partially ordered set $(\bZ_+, \leq) $. We consider the following data (``system''):
\begin{enumerate}
 \item Categories $\mathcal{C}_i$ for each $i \in \bZ_+$.
 \item Functors $\mathcal{F}_{i-1, i}: \mathcal{C}_i \rightarrow \mathcal{C}_{i-1}$ for each $i \geq 1$. 
%
\end{enumerate}

\begin{definition}
Given the above data, we define the inverse limit category $\varprojlim_{i \in \bZ_+} \mathcal{C}_i$  to be the following category:
\begin{itemize}
 \item The objects are pairs $(\{C_i\}_{i \in \bZ_+}, \{\phi_{i-1, i}\}_{i \geq 1})$ where $C_i \in \mathcal{C}_i$ for each $i \in \bZ_+$ and $\phi_{i-1, i}: \mathcal{F}_{i-1, i}(C_i) \stackrel{\sim}{\rightarrow} C_{i-1}$ for any $i \geq 1$.
\item A morphism $f$ between two objects $(\{C_i\}_{i \in \bZ_+}, \{\phi_{i-1, i}\}_{i \geq 1})$, $(\{D_i\}_{i \in \bZ_+}, \{\psi_{i-1, i}\}_{i \geq 1})$ is a set of arrows $\{f_i: C_i \rightarrow D_i\}_{i \in \bZ_+}$ such that for any $i \geq 1$, the following diagram is commutative:
$$ \begin{CD}
    \mathcal{F}_{i-1, i}(C_i) @>{\phi_{i-1, i}}>> C_{i-1}\\
    @V{\mathcal{F}_{i-1, i}(f_i)}VV @V{f_{i-1}}VV \\
    \mathcal{F}_{i-1, i}(D_i) @>{\psi_{i-1, i}}>> D_{i-1}
   \end{CD}$$
 Composition of morphisms is component-wise.
\end{itemize}

\end{definition}

The definition of $\varprojlim_{i \in \bZ_+} \mathcal{C}_i$ implies that for each $i \in \bZ_+$, we can define functors
\begin{align*}
\mathbf{Pr}_i: \varprojlim_{i \in \bZ_+} \mathcal{C}_i &\rightarrow \mathcal{C}_i \\
C=(\{C_i\}_{i \in \bZ_+}, \{\phi_{i-1, i}\}_{i \geq 1})) &\mapsto C_i \\
f=\{f_i: C_i \rightarrow D_i\}_{i \in \bZ_+}  &\mapsto f_i
\end{align*}
which satisfy the following property (this property follows directly from the definition of $\varprojlim_{i \in \bZ_+} \mathcal{C}_i$): 
\begin{lemma}
 For any $i \geq 1$, $\mathcal{F}_{i-1, i} \circ \mathbf{Pr}_i \cong \mathbf{Pr}_{i-1}$, with a natural isomorphism given by:
 $$(\mathcal{F}_{i-1, i} \circ \mathbf{Pr}_i)(C) \stackrel{\phi_{i-1, i}}{\rightarrow} \mathbf{Pr}_{i-1}(C)$$
 (here $C= (\{C_i\}_{i \in \bZ_+}, \{\phi_{i-1, i}\}_{i \geq 1}))$).
\end{lemma}

Let $\mathcal{A}$ be a category, together with a set of functors $\mathcal{G}_i: \mathcal{A} \rightarrow \mathcal{C}_i$ which satisfy: for any $i  \geq 1$, there exists a natural isomorphism
$$ \eta_{i-1, i}:\mathcal{F}_{i-1, i} \circ \mathcal{G}_i \rightarrow \mathcal{G}_{i-1} $$

Then $\varprojlim_{i \in \bZ_+} \mathcal{C}_i$ is universal among such categories; that is, we have a functor 
\begin{align*}
\mathcal{G}: \mathcal{A} &\rightarrow \varprojlim_{i \in \bZ_+} \mathcal{C}_i \\
A &\mapsto (\{\mathcal{G}_i(A)\}_{i \in \bZ_+}, \{\eta_{i-1, i}\}_{i \geq 1}) \\
(f: A_1 \rightarrow A_2 ) &\mapsto \{f_i:= \mathcal{G}_i(f)\}_{i \in \bZ_+} 
\end{align*}
and $\mathcal{G}_i \cong \mathbf{Pr}_i \circ \mathcal{G}$ for every $i \in \bZ_+$.

Finally, we give the following simple lemma:
\begin{lemma}\label{lem:inv_limit_stabilizing_cat}
Let $N \in \bZ_+$, and assume that for any $i \geq N$, $\mathcal{F}_{i-1, i}$ is an equivalence. Then $\mathbf{Pr}_i: \varprojlim_{j \in \bZ_+} \mathcal{C}_j \rightarrow \mathcal{C}_i$ is an equivalence for any $i \geq N$.
\end{lemma}
\begin{proof}
Set $\mathcal{F}_{ij} := \mathcal{F}_{i, i+1} \circ ... \circ  \mathcal{F}_{j-1, j}$ for any $i \leq j$ (in particular, $\mathcal{F}_{ii} := \id_{\mathcal{C}_i}$).

Fix $i \geq N$. Let $j \geq i$; then $\mathcal{F}_{ij}$ is an equivalence, i.e. we can find a functor $$\mathcal{G}_{j}: \mathcal{C}_i \rightarrow \mathcal{C}_j$$ such that $\mathcal{F}_{ij} \circ \mathcal{G}_{j} \cong \id_{\mathcal{C}_i}$, and $\mathcal{G}_{j} \circ \mathcal{F}_{ij} \cong \id_{\mathcal{C}_j}$ (for $j :=i$, we put $\mathcal{G}_{i}:= \id_{\mathcal{C}_i}$).

For any $j > i$, fix natural transformations $$\eta_{j-1, j}:\mathcal{F}_{j-1, j} \circ  \mathcal{G}_{j} \stackrel{\sim}{\rightarrow} \mathcal{G}_{j-1}$$

For any $j \leq i$, put $\mathcal{G}_{j}:= \mathcal{F}_{ji}$, and $\eta_{j-1, j} := \id$.

Then the universal property of $\varprojlim_{j \in \bZ_+} \mathcal{C}_j$ implies that there exists a functor $$\mathcal{G}: \mathcal{C}_i \rightarrow \varprojlim_{j \in \bZ_+} \mathcal{C}_j$$ such that $\mathbf{Pr}_j \circ \mathcal{G} \cong \mathcal{G}_{j}$ for any $j$.
The functor $\mathcal{G}$ is given by
\begin{align*}
\mathcal{G}: \mathcal{C}_i &\rightarrow \varprojlim_{j \in \bZ_+} \mathcal{C}_j \\
C &\mapsto (\{\mathcal{G}_j(C)\}_{j \in \bZ_+}, \{\eta_{j-1, j}\}_{j \geq 1}) \\
f: C \rightarrow C' &\mapsto \{f_j:= \mathcal{G}_j(f)\}_{j \in \bZ_+} 
\end{align*}

In particular, we have: $\mathbf{Pr}_i \circ\mathcal{G} \cong \id_{\mathcal{C}_i}$. It remains to show that $\mathcal{G} \circ \mathbf{Pr}_i \cong \id_{\varprojlim_{j \in \bZ_+} \mathcal{C}_j}$, and this will prove that $\mathbf{Pr}_i$ is an equivalence of categories.

For any $C \in \varprojlim_{j \in \bZ_+} \mathcal{C}_j$, $C: =(\{C_j\}_{j \in \bZ_+}, \{\phi_{j-1, j}\}_{j \geq 1})$, and for any $l\leq j$ we define isomorphisms $\phi_{lj}: \mathcal{F}_{lj}(C_j) \rightarrow C_l$ given by 
$$\phi_{lj}:= \phi_{l, l+1} \circ \mathcal{F}_{l, l+1}( \phi_{l+1, l+2} \circ \mathcal{F}_{l+1, l+2}( \phi_{l+2, l+3} \circ... \circ \mathcal{F}_{j-2, j-1}(\phi_{j-1, j})...)) $$ Define
$\theta(C) := \{\theta(C)_j: C_j \rightarrow \mathbf{Pr}_j(\mathcal{G}(C_i)) \cong \mathcal{G}_j(C_i)\}_{j \in I}$ by setting $$\theta(C)_j := \begin{cases}
          \phi_{\InnaA{ji}}^{-1} &\text{ if } j \leq i \\
          \mathcal{G}_{j}(\phi_{ij}) &\text{ if } j > i                                                                                                                                  \end{cases}
$$

Now, let $C: =(\{C_j\}_{j \in \bZ_+}, \{\phi_{j-1, j}\}_{j \geq 1})$, $D:=(\{D_j\}_{j \in \bZ_+}, \{\psi_{j-1, j}\}_{j \geq 1})$ be objects in $\varprojlim_{j \in \bZ_+} \mathcal{C}_j$, together with a morphism $f: C \rightarrow D$, $f:=\{f_j: C_j \rightarrow D_j\}_{j \in \bZ_+}$.

Then the diagram 
$$\begin{CD}
   C @>{\theta(C)}>> (\mathcal{G} \circ \mathbf{Pr}_i)(C) \\
   @V{f}VV @V{(\mathcal{G} \circ \mathbf{Pr}_i)(f)}VV \\
   D @>{\theta(D)}>> (\mathcal{G}  \circ \mathbf{Pr}_i)(D)
  \end{CD}
$$
is commutative, since for $j \leq i$, the diagrams
$$\begin{CD}
   C_j @>{\phi_{ji}^{-1}}>> \mathbf{Pr}_j(\mathcal{G}(C_i)) \cong \mathcal{G}_j(C_i) \\
   @V{f_j}VV @V{\mathcal{G}_j(f_i)}VV \\
   D_j @>{\psi_{ji}^{-1}}>> \mathbf{Pr}_j(\mathcal{G}(D_i)) \cong \mathcal{G}_j(D_i)
  \end{CD}
$$
are commutative, and for $j > i$, the diagrams
$$\begin{CD}
   C_j @>{\mathcal{G}_j(\phi_{ij})}>> \mathbf{Pr}_j(\mathcal{G}(C_i))  \cong \mathcal{G}_j(C_i)\\
   @V{f_j}VV @V{\mathcal{G}_j(f_i)}VV \\
   D_j @>{\mathcal{G}_j(\psi_{ij})}>> \mathbf{Pr}_j(\mathcal{G}(D_i))  \cong \mathcal{G}_j(D_i)
  \end{CD}
$$
are commutative.

\end{proof}

\subsection{Inverse limit of pre-additive, additive and abelian categories}\label{subsec:abel_inverse_limit}
In this subsection, we give some more or less trivial properties of the inverse limit corresponding to the system $((\mathcal{C}_{i})_{i \in \bZ_+}, (\mathcal{F}_{i-1, i})_{i \geq 1})$ depending on the properties of the categories $\mathcal{C}_{i}$ and the functors $\mathcal{F}_{i-1, i}$.

\begin{lemma}
 Assume the categories $\mathcal{C}_i$ are $\bC$-linear pre-additive categories (i.e. the $\Hom$-spaces in each $\mathcal{C}_i$ are complex vector spaces), and the functors $\mathcal{F}_{i-1, i}$ are $\bC$-linear. Then the category $\varprojlim_{i \in \bZ_+} \mathcal{C}_i$ is automatically a $\bC$-linear pre-additive category:
 
 given $f,g : C \rightarrow D$ in $\varprojlim_{i \in \bZ_+} \mathcal{C}_i$, where $C= (\{C_i\}_{i \in \bZ_+}, \{\phi_{i-1, i}\}_{i \geq 1})$, $D=(\{D_i\}_{i \in \bZ_+}, \{\psi_{i-1, i}\}_{i \geq 1})$, $f= \{f_i: C_i \rightarrow D_i\}_{i \in \bZ_+}, g= \{g_i: C_i \rightarrow D_i\}_{i \in \bZ_+}$, we have:
 $$ \alpha f+ \beta g := \{(\alpha f_i +\beta g_i): C_i \rightarrow D_i\}_{i \in \bZ_+}$$
 where $\alpha, \beta \in \bC$.
 
 The functors $\mathbf{Pr}_i$ are then $\bC$-linear.
\end{lemma}

\begin{lemma}
 Assume the categories $\mathcal{C}_i$ are additive categories (i.e. each $\mathcal{C}_i$ is pre-additive and has biproducts), and the functors $\mathcal{F}_{i-1, i}$ are additive. Then the category $\varprojlim_{i \in \bZ_+} \mathcal{C}_i$ is automatically a additive category:
 \begin{itemize}
  \item The zero object in $\varprojlim_{i \in \bZ_+} \mathcal{C}_i$ is $(\{0_{\mathcal{C}_i}\}_{i \in \bZ_+}, \{0\}_{i \geq 1})$.
  \item Given $ C , D$ in $\varprojlim_{i \in \bZ_+} \mathcal{C}_i$, where $C= (\{C_i\}_{i \in \bZ_+}, \{\phi_{i-1, i}\}_{i  \geq 1})$, $D=(\{D_i\}_{i \in \bZ_+}, \{\psi_{i-1, i}\}_{i \geq 1})$, we have:
 $$ C \oplus D := (\{(C_i \oplus D_i\}_{i \in \bZ_+} , \{\phi_{i-1, i} \oplus \psi_{i-1, i}\}_{i \geq 1})$$
 with obvious inclusion and projection maps.
 \end{itemize}

The functors $\mathbf{Pr}_i$ are then additive.
\end{lemma}
\begin{proof}
 Let $X, Y \in \varprojlim_{i \in \bZ_+} \mathcal{C}_i$, $X= (\{X_i\}_{i \in \bZ_+}, \{\mu_{i-1, i}\}_{i \geq 1})$, $Y=(\{Y_i\}_{i \in \bZ_+}, \{\rho_{i-1, i}\}_{i \geq 1})$, and let $f_C: X \rightarrow C$, $f_D: X \rightarrow D$, $g_C: C \rightarrow Y$, $g_D: D \rightarrow Y$ (we denote the components of the map $f_C$ by $f_{C_i}$, of the map $f_D$ by $f_{D_i}$, etc.). 
 
 Denote by $\iota_{C_i}, \iota_{D_i}, \pi_{C_i}, \pi_{D_i}$ the inclusion and projection maps between $C_i, D_i$ and $C_i \oplus D_i$. By definition, $\iota_C := \{\iota_{C_i}\}_{i \in \bZ_+}, \iota_D := \{\iota_{D_i}\}_{i \in \bZ_+}, \pi_C := \{\pi_{C_i}\}_{i \in \bZ_+}, \pi_D := \{\pi_{D_i}\}_{i \in \bZ_+}$ are the inclusion and projection maps between $C, D$ and $C \oplus D$.
 
 For each $i$, there exists a unique map $f_i: X_i \rightarrow C_i \oplus D_i$ and a unique map $g_i:  C_i \oplus D_i  \rightarrow Y_i$ such that $$ \pi_{C_i} \circ f_i = f_{C_i}, \pi_{D_i} \circ f_i =f_{D_i}, g_i \circ \iota_{C_i} = g_{C_i}, g_i \circ \iota_{D_i} = g_{D_i}$$ for any $i \in \bZ_+$.
 
 This means that we have a unique map $f: X \rightarrow C \oplus D$ and a unique map $g: C \oplus D \rightarrow Y$ such that $$\pi_{C} \circ f = f_{C}, \pi_{D} \circ f =f_{D}, g \circ \iota_{C} = g_{C}, g \circ \iota_{D} = g_{D}$$
 (these are the maps $f = \{f_i\}_i, g= \{g_i\}_i$).
 
\end{proof}

\begin{lemma}\label{lem:inv_limit_cat_isom}
 Let $f : C \rightarrow D$ in $\varprojlim_{i \in \bZ_+} \mathcal{C}_i$, where $C= (\{C_i\}_{i \in \bZ_+}, \{\phi_{i-1, i}\}_{i \geq 1})$, $D=(\{D_i\}_{i \in \bZ_+}, \{\psi_{i-1, i}\}_{i \geq 1})$, $f= \{f_i: C_i \rightarrow D_i\}_{i \in \bZ_+}$.
 
 Assume $f_i$ are isomorphisms for each $i$. Then $f$ is an isomorphism.
\end{lemma}
\begin{proof}
 Let $g_i := f_i^{-1}$ for each $i \in \bZ_+$ (this morphism exists since $f_i$ is an isomorphism, and is unique). All we need is to show that $g := \{ g_i: D_i \rightarrow C_i\}_{i}$ is a morphism from $D$ to $C$ in $\varprojlim_{i \in \bZ_+} \mathcal{C}_i$, i.e. that the following diagram is commutative for any $i \geq 1$:
 $$ \begin{CD}
    \mathcal{F}_{i-1, i}(C_i) @>{\phi_{i-1, i}}>> C_{i-1}\\
    @A{\mathcal{F}_{i-1, i}(g_i)}AA @A{g_{i-1}}AA \\
    \mathcal{F}_{i-1, i}(D_i) @>{\psi_{i-1, i}}>> D_{i-1}
   \end{CD}$$
   The morphism $g_{i-1} \circ \psi_{i-1, i}$ is inverse to $\psi_{i-1, i}^{-1} \circ f_{i-1}$, and $\phi_{i-1, i} \circ \mathcal{F}_{i-1, i}(g_ji) $ is inverse to $\mathcal{F}_{i-1, i}(f_i) \circ \phi_{i-1, i}^{-1}$.
   
   But $\psi_{i-1, i}^{-1} \circ f_{i-1} = \mathcal{F}_{i-1, i}(f_i) \circ \phi_{i-1, i}^{-1}$, since $f= \{f_i: C_i \rightarrow D_i\}_{i \in \bZ_+}$ is a morphism from $C$ to $D$ in $\varprojlim_{i \in \bZ_+} \mathcal{C}_i$.
   The uniqueness of the inverse morphism then implies that $g_{i-1} \circ \psi_{i-1, i} = \phi_{i-1, i} \circ \mathcal{F}_{i-1, i}(g_i) $, and we are done.
\end{proof}

\begin{proposition}
 Assume the categories $\mathcal{C}_i$ are abelian, and the functors $\mathcal{F}_{i-1, i}$ are exact. Then the category $\varprojlim_{i \in \bZ_+} \mathcal{C}_i$ is automatically abelian:
\begin{itemize}
 \item Given $f : C \rightarrow D$ in $\varprojlim_{i \in \bZ_+} \mathcal{C}_i$, where $C= (\{C_i\}_{i \in \bZ_+}, \{\phi_{i-1, i}\}_{i \geq 1})$, $D=(\{D_i\}_{i \in \bZ_+}, \{\psi_{i-1, i}\}_{i  \geq 1})$, $f= \{f_i: C_i \rightarrow D_i\}_{i \in \bZ_+}$, $f$ has a kernel and a cokernel:
$$Ker(f) := (\{Ker(f_i)\}_{i \in \bZ_+}, \{ \rho_{i-1, i} \}_{i  \geq 1}), Coker(f) := (\{Coker(f_i)\}_{i \in \bZ_+}, \{ \mu_{i-1, i} \}_{i  \geq 1})$$
where $\rho_{i-1, i}, \mu_{i-1, i}$ are the unique maps making the following diagram commutative:
$$ \begin{CD}
Ker(\mathcal{F}_{i-1, i}(f_i))\cong \mathcal{F}_{i-1, i}(Ker(f_i)) @>{\rho_{i-1, i}}>> Ker(f_{i-1})\\
    @VVV @VVV \\
    \mathcal{F}_{i-1, i}(C_i) @>{\phi_{i-1, i}}>> C_{i-1}\\
    @V{\mathcal{F}_{ij}(f_i)}VV @V{f_{i-1}}VV \\
    \mathcal{F}_{i-1, i}(D_i) @>{\psi_{i-1, i}}>> D_{i-1} \\
    @VVV @VVV \\
    Coker(\mathcal{F}_{i-1, i}(f_i)) \cong \mathcal{F}_{i-1, i}(Coker(f_i)) @>{\mu_{i-1, i}}>> Coker(f_{i-1})
   \end{CD}
$$
 \item Given $f : C \rightarrow D$ in $\varprojlim_{i \in \bZ_+} \mathcal{C}_i$, we have: $Im(f) := Ker(Coker(f)) \cong Coker (Ker (f)) =: Coim(f)$.
\end{itemize}

\end{proposition}
\begin{proof}
%

The universal properties of $Ker(f), Coker(f)$ hold automatically, as a consequence of the universal properties of $Ker(f_i), Coker(f_i)$.

Now, let $f : C \rightarrow D$ in $\varprojlim_{i \in \bZ_+} \mathcal{C}_i$, where $C= (\{C_i\}_{i \in \bZ_+}, \{\phi_{i-1, i} \}_{i \geq 1} )$, $D=(\{D_i\}_{i \in \bZ_+}, \{\psi_{i-1, i} \}_{i \geq 1})$, $f= \{f_i: C_i \rightarrow D_i\}_{i \in \bZ_+}$. 

Consider the objects $Im(f) := Ker(Coker(f)), Coim(f):= Coker (Ker (f)) $ in $\varprojlim_{i \in \bZ_+} \mathcal{C}_i$.
We have a canonical map $\bar{f}: Coim(f) \rightarrow Im(f)$, such that $f : C \rightarrow D$ is the composition
$$C \twoheadrightarrow Coim(f) \stackrel{\bar{f}}{\longrightarrow} Im(f) \hookrightarrow D $$

Consider the maps $\bar{f}_i$ for each $i \in \bZ_+$, where $\bar{f}_i$ is the canonical map such that $f_i : C_i \rightarrow D_i$ is the composition
$$C_i \twoheadrightarrow Coim(f_i) \stackrel{\bar{f_i}}{\longrightarrow} Im(f_i) \hookrightarrow D_i $$
One then immediately sees that $\bar{f}= \{\bar{f}_i: Coim(f_i) \rightarrow Im(f_i)\}_i$. 

Since the category $\mathcal{C}_i$ is abelian for each $i \in \bZ_+$, the map $\bar{f}_i$ is an isomorphism. Lemma \ref{lem:inv_limit_cat_isom} then implies that $\bar{f}$ is an isomorphism as well.
\end{proof}

The following is a trivial corollary of the previous proposition:
\begin{corollary}\label{cor:inv_limit_cat_exact}
The functors $\mathbf{Pr}_i$ are exact. 
%
%
\end{corollary}

This corollary, in turn, immediately implies the following statement:

\begin{corollary}\label{cor:prop_functor_to_inv_lim}
Let $(\mathcal{C}_{i}, \mathcal{F}_{ij})$ be a system of pre-additive (respectively, additive, abelian) categories, and linear (respectively, additive, exact) functors.

Let $\mathcal{A}$ be a pre-additive (respectively, additive, abelian) category, together with a set of linear (respectively, additive, exact) functors $\mathcal{G}_i: \mathcal{A} \rightarrow \mathcal{C}_i$ which satisfy: for any $i \geq 1$, there exists a natural isomorphism
$$ \eta_{i-1, i}:\mathcal{F}_{i-1, i} \circ \mathcal{G}_i \rightarrow \mathcal{G}_{i-1} $$

Then $\varprojlim_{i \in \bZ_+} \mathcal{C}_i$ is universal among such categories; that is, we have a linear (respectively, additive, exact) functor 
\begin{align*}
\mathcal{G}: \mathcal{A} &\rightarrow \varprojlim_{i \in \bZ_+} \mathcal{C}_i \\
A &\mapsto (\{\mathcal{G}_i(A)\}_{i \in \bZ_+}, \{\eta_{i-1, i}\}_{i \in \bZ_+}) \\
f: A_1 \rightarrow A_2 &\mapsto \{f_i:= \mathcal{G}_i(f)\}_{i \in \bZ_+} 
\end{align*}
and $\mathcal{G}_i \cong \mathbf{Pr}_i \circ \mathcal{G}$ for every $i \in \bZ_+$.
\end{corollary}

\section{Restricted inverse limit of finite-length categories}\label{sec:stab_inv_lim}
\subsection{}
We consider the case when the categories $\mathcal{C}_i$ are finite-length. We would like to give a notion of an inverse limit of the system $((\mathcal{C}_i)_{i\in \bZ_+}, (\mathcal{F}_{i-1, i})_{i \geq 1})$ which would be a finite-length category as well. In order to do this, we will define the notion of a ``shortening'' functor, and define a ``stable'' inverse limit of a system of finite-length categories and shortening functors.

\begin{definition}\label{def:shorten_functor}
 Let $\mathcal{A}_1, \mathcal{A}_2$ be finite-length categories. An exact functor $\mathcal{F}:\mathcal{A}_1 \longrightarrow \mathcal{A}_2$ will be called {\it shortening} if for any object $A \in \mathcal{A}_1$, we have:
 $$\ell_{\mathcal{A}_1}(A) \geq \ell_{\mathcal{A}_2}(\mathcal{F}(A))$$
\end{definition}

Since $\mathcal{F}$ is exact, this is equivalent to requiring that for any simple object $L \in \mathcal{A}_1$, the object $\mathcal{F}(L)$ is either simple or zero.

\begin{definition}\label{def:stab_inv_lim}
 Let $((\mathcal{C}_i)_{i\in \bZ_+}, (\mathcal{F}_{i-1, i})_{i \geq 1})$ be a system of finite-length categories and shortening functors. We will denote by $\varprojlim_{i \in \bZ_+, \text{ restr}} \mathcal{C}_i$ the full subcategory of $\varprojlim_{i \in \bZ_+} \mathcal{C}_i$ whose objects $C: =(\{C_j\}_{j \in \bZ_+}, \{\phi_{j-1, j}\}_{j \geq 1})$ satisfy:
 the integer sequence $\{ \ell_{\mathcal{C}_i}(C_i) \}_{i \geq 0}$ stabilizes.
\end{definition}

Note that the since the functors $\mathcal{F}_{i-1, i}$ are shortening, the sequence $\{ \ell_{\mathcal{C}_i}(C_i) \}_{i \geq 0}$ is weakly increasing. Therefor, this sequence stabilizes iff it is bounded from above.

We now show that $\varprojlim_{i \in \bZ_+, \text{ restr}} \mathcal{C}_i$ is a finite-length category.  
\begin{lemma}\label{lem:stab_lim_is_artinian}
 The category $\mathcal{C}:= \varprojlim_{i \in \bZ_+, \text{ restr}} \mathcal{C}_i$ is a Serre subcategory of $\varprojlim_{i \in \bZ_+} \mathcal{C}_i$, and its objects have finite length.
 
 Moreover, given an object $C : =(\{C_i\}_{i \in \bZ_+}, \{\phi_{i-1, i}\}_{i \geq 1})$ in $\mathcal{C}$, we have: 
 $$ \ell_{\mathcal{C}} (C) \leq \max \{ \ell_{\mathcal{C}_i} (C_i) \rvert i \geq 0 \}$$
\end{lemma}
\begin{proof}
 Let $$C : =(\{C_j\}_{j \in \bZ_+}, \{\phi_{j-1, j}\}_{j \geq 1}), \; C' : =(\{C'_j\}_{j \in \bZ_+}, \{\phi'_{j-1, j}\}_{j \geq 1}), \; C'' : =(\{C''_j\}_{j \in \bZ_+}, \{\phi''_{j-1, j}\}_{j \geq 1})$$ be objects in $\varprojlim_{i \in \bZ_+} \mathcal{C}_i$, together with morphisms $f: C' \rightarrow C$, $g: C \rightarrow C''$ such that the sequence $$ 0 \rightarrow C' \stackrel{f}{\longrightarrow} C \stackrel{g}{\longrightarrow} C'' \rightarrow 0$$ is exact.
 
 If $C$ lies in the subcategory $\mathcal{C}$, then the sequence $\{ \ell_{\mathcal{C}_i}(C_i) \}_{i \geq 0}$ is bounded from above, and stabilizes. Denote its maximum by $N$. For each $i$, the sequence
 $$ 0 \rightarrow C'_i \stackrel{f_i}{\longrightarrow} C_i \stackrel{g}{\longrightarrow} C''_i \rightarrow 0$$ is exact. Therefore, $ \ell_{\mathcal{C}_i}(C'_i),  \ell_{\mathcal{C}_i}(C''_i) \leq N$ for each $i$, and thus $C', C''$ lie in $\mathcal{C}$ as well.
 
Vice versa, assuming $C', C''$ lie in $\mathcal{C}$, denote by $N', N''$ the maximums of the sequences $\{ \ell_{\mathcal{C}_i}(C'_i) \}_i,  \{ \ell_{\mathcal{C}_i}(C''_i) \}_i$ respectively. Then $\ell_{\mathcal{C}_i}(C_i) \leq N' + N''$ for any $i \geq 0$, and so $C$ lies in the subcategory $\mathcal{C}$ as well.

Thus $\mathcal{C}$ is a Serre subcategory of $\varprojlim_{i \in \bZ_+} \mathcal{C}_i$.

Next, let $C$ lie in $\mathcal{C}$. We would like to say that $C$ has finite length. Denote by $N$ the maximum of the sequence $\{ \ell_{\mathcal{C}_i}(C_i) \}_{i \geq 0}$. It is easy to see that $C$ has length at most $N$; indeed, if $\{C', C'', ..., C^{(n)} \}$ is a subset of $JH_{\mathcal{C}}(C)$, then for some $i >>0$, we have: $\mathbf{Pr}_i(C^{(k)}) \neq 0$ for any $k=1, 2, ..., n$. $\mathbf{Pr}_i(C^{(k)})$ are distinct Jordan Holder components of $C_i$, so $ n \leq \ell_{\mathcal{C}_i}(C_i) \leq N$. In particular, we see that
$$ \ell_{\mathcal{C}} (C) \leq N = \max \{ \ell_{\mathcal{C}_i} (C_i) \rvert i \geq 0 \}$$

\end{proof}

\begin{notation}
 Denote by $Irr(\mathcal{C}_i)$ the set of isomorphism classes of irreducible objects in $\mathcal{C}_i$, and define the pointed set $$Irr_*(\mathcal{C}_i) := Irr(\mathcal{C}_i) \sqcup \{0\}$$

The shortening functors $\mathcal{F}_{i-1, i}$ then define maps of pointed sets $$f_{i-1, i}: Irr_*(\mathcal{C}_i) \longrightarrow Irr_*(\mathcal{C}_{i-1})$$

Similarly, we define $Irr \left( \varprojlim_{i \in \bZ_+, \text{ restr}} \mathcal{C} \right)$ to be the set of isomorphism classes of irreducible objects in $\mathcal{C}$, and define the pointed set $$Irr_*(\mathcal{C}) := Irr(\mathcal{C}) \sqcup \{0\}$$

\end{notation}

Let $C : =(\{C_j\}_{j \in \bZ_+}, \{\phi_{j-1, j}\}_{j \geq 1})$ be an object in $\mathcal{C}$. We denote by $JH(C_j)$ the multiset of the Jordan-Holder components of $C_j$, and let $$JH_*(C_j) := JH(C_j) \sqcup \{0\}$$ The corresponding set lies in $ Irr_*(\mathcal{C}_j)$, and we have maps of (pointed) multisets $$f_{j-1, j}: JH_*(C_j) \rightarrow JH_*(C_{j-1})$$

\mbox{}

Denote by $\varprojlim_{i \in \bZ_+} Irr_*(\mathcal{C}_i)$ the inverse limit of the system $( \{Irr_*(\mathcal{C}_i \}_{i \geq 0}, \{f_{i-1, i} \}_{i \geq 1})$. We will also denote by $pr_j: \varprojlim_{i \in \bZ_+} Irr_*(\mathcal{C}_i) \rightarrow Irr_*(\mathcal{C}_j)$ the projection maps. 

The elements of the set $\varprojlim_{i \in \bZ_+} Irr_*(\mathcal{C}_i)$ are just sequences $(L_i)_{i  \geq 0}$ such that $L_i \in Irr_*(\mathcal{C}_i)$, and $f_{i-1, i}(L_i) \cong L_{i-1}$. 
%
%

The following lemma describes the simple objects in the category $\mathcal{C} := \varprojlim_{i \in \bZ_+, \text{ restr}} \mathcal{C}_i$. 


\begin{lemma}\label{lem:simple_obj_stab_lim}
 Let $C : =(\{C_j\}_{j \in \bZ_+}, \{\phi_{j-1, j}\}_{j \geq 1})$ be an object in $\mathcal{C}:=\varprojlim_{i \in \bZ_+, \text{ restr}} \mathcal{C}_i$.
 
 Then $$C \in Irr_*(\mathcal{C}) \; \; \Longleftrightarrow \; \;  \mathbf{Pr}_j(C)= C_j \in Irr_*(\mathcal{C}_j)  \; \forall j $$
 
 In other words, $C$ is a simple object (that is, $C$ has exactly two distinct subobjects: zero and itself) iff $C \neq 0$, and for any $j \geq 0$, the component $C_j$ is either a simple object in $\mathcal{C}_j$, or zero.
\end{lemma}
\begin{proof}
 
 The direction $\Leftarrow$ is obvious, so we will only prove the direction $\Rightarrow$.
 
 Let $n_0$ be a position in which the maximum of the weakly-increasing integer sequence $ \{ \ell_{\mathcal{C}_i}(C_i) \}_{i \geq 0}$ is obtained. By definition of $n_0$, for $j > n_0$, the functors $\mathcal{F}_{j-1, j}$ do not kill any Jordan-Holder components of $C_j$. 
 
 Now, consider the socles of the objects $C_j$ for $j \geq n_0$. For any $j >0$, we have: $$\mathcal{F}_{j-1, j}(\text{ } socle(C_j) \text{ }) \stackrel{\phi_{j-1, j}}{\hookrightarrow} \text{ } socle(C_{j-1})$$ and thus for $j > n_0$, we have $$\ell_{\mathcal{C}_j}(\text{ } socle(C_j) \text{ }) = \ell_{\mathcal{C}_{j-1}}(\mathcal{F}_{j-1, j}(\text{ } socle(C_j) \text{ }))  \leq \ell_{\mathcal{C}_{j-1}}(\text{ } socle(C_{j-1}) \text{ })$$ 
 Thus the sequence $$\{\ell_{\mathcal{C}_j}(\text{ } socle(C_j) \text{ }) \}_{j \geq n_0}$$ is a weakly decreasing sequence, and stabilizes. Denote its stable value by $N$. We conclude that there exists $n_1 \geq n_0$ so that $$\mathcal{F}_{j-1, j}(\text{ } socle(C_j) \text{ }) \stackrel{\phi_{j-1, j}}{\longrightarrow} \text{ } socle(C_{j-1})$$ is an isomorphism for every $j > n_1$.
 
 Now, denote: $$D_j := \begin{cases}
                        \mathcal{F}_{j, n_1}(\text{ } socle(C_{n_1}) \text{ }) &\text{ if } j < n_1 \\
                        socle(C_j) &\text{ if } j \geq n_1 
                       \end{cases}
$$
and we put: $D:= ((D_j)_{j \geq 0}, (\phi_{j-1, j})_{j \geq 1})$ (this is a subobject of $C$ in the category $\varprojlim_{i \in \bZ_+} \mathcal{C}_i$). Of course, $\ell_{\mathcal{C}_j}(D_j) \leq N$ for any $j$, so $D$ is an object in the full subcategory $\mathcal{C}$ of $\varprojlim_{i \in \bZ_+} \mathcal{C}_i$. 

Furthermore, since $C \neq 0$, we have: for $j >>0$, $socle(C_j) \neq 0$, and thus $0 \neq D \subset C$.

$D$ is a semisimple object $\mathcal{C}$, with simple summands corresponding to the elements of the inverse limit of the multisets $\varprojlim_{ j \in \bZ_+} JH_*(D_j)$. 

We conclude that $D=C$, and that $socle(C_j) = C_j$ has length at most one for any $j \geq 0$.
\begin{remark}
Note that the latter multiset is equivalent to the inverse limit of multisets $JH_*( \text{ } socle(C_j) \text{ } )$, so $D$ is, in fact, the socle of $C$.  
\end{remark}

\end{proof}
\begin{corollary}\label{cor:param_simple_obj_stab_lim}
 The set of isomorphism classes of simple objects in $\varprojlim_{i \in \bZ_+, \text{ restr}} \mathcal{C}_i$ is in bijection with the set  $\varprojlim_{i \in \bZ_+} Irr_*(\mathcal{C}_i) \setminus \{0\}$. That is, we have a natural bijection
 $$Irr_*(\mathcal{C}) \cong \varprojlim_{i \in \bZ_+} Irr_*(\mathcal{C}_i)$$
\end{corollary}

In particular, given an object $C : =(\{C_j\}_{j \in \bZ_+}, \{\phi_{j-1, j}\}_{j \geq 1})$ in $\varprojlim_{i \in \bZ_+, \text{ restr}} \mathcal{C}_i$, we have: $JH_*(C) = \varprojlim_{ i\in \bZ_+} JH_*(C_i)$ (an inverse limit of the system of multisets $JH_*(C_j)$ and maps $f_{j-1, j}$).

It is now obvious that the projection functors $\mathbf{Pr}_i$ are shortening as well:

\begin{corollary}\label{cor:proj_funct_shorten}
 The projection functors $\mathbf{Pr}_i$ are shortening, and define the maps $$pr_i: Irr_*(\mathcal{C}) \longrightarrow Irr_*(\mathcal{C}_i)$$
\end{corollary}

Lemma \ref{lem:stab_lim_is_artinian} and Corollary \ref{cor:proj_funct_shorten} give us:
\begin{corollary}\label{cor:length_obj_stab_inv_lim}
  Given an object $C : =(\{C_i\}_{i \in \bZ_+}, \{\phi_{i-1, i}\}_{i \geq 1})$ in $\mathcal{C}$, we have: 
 $$ \ell_{\mathcal{C}} (C) = \max \{ \ell_{\mathcal{C}_i} (C_i) \rvert i \geq 0 \}$$
\end{corollary}

It is now easy to see that the restricted inverse limit has the following universal property:

\begin{proposition}\label{prop:stab_lim_univ_prop}
 Let $\mathcal{A}$ be a finite-length category, together with a set of shortening functors $\mathcal{G}_i: \mathcal{A} \rightarrow \mathcal{C}_i$ which satisfy: for any $i  \geq 1$, there exists a natural isomorphism
$$ \eta_{i-1, i}:\mathcal{F}_{i-1, i} \circ \mathcal{G}_i \rightarrow \mathcal{G}_{i-1} $$
Then $\varprojlim_{i \in \bZ_+, \text{ restr}} \mathcal{C}_i$ is universal among such categories; that is, 
we have a shortening functor 
\begin{align*}
\mathcal{G}: \mathcal{A} &\rightarrow \varprojlim_{i \in \bZ_+, \text{ restr}} \mathcal{C}_i \\
A &\mapsto (\{\mathcal{G}_i(A)\}_{i \in \bZ_+}, \{\eta_{i-1, i}\}_{i \geq 1}) \\
f: A_1 \rightarrow A_2 &\mapsto \{f_i:= \mathcal{G}_i(f)\}_{i \in \bZ_+} 
\end{align*}
and $\mathcal{G}_i \cong \mathbf{Pr}_i \circ \mathcal{G}$ for every $i \in \bZ_+$.

\end{proposition}
\begin{proof}
 Consider the functor $\mathcal{G}: \mathcal{A} \rightarrow \varprojlim_{i \in \bZ_+} \mathcal{C}_i$ induced by the functors $\mathcal{G}_i$. We would like to say that for any $A \in \mathcal{A}$, the object $\mathcal{G}(A)$ lies in the subcategory $\varprojlim_{i \in \bZ_+, \text{ restr}} \mathcal{C}_i$, i.e. that the sequence $\{ \ell_{ \mathcal{C}_i} (\mathcal{G}_i(A)) \}_i$ is bounded from above. 
 
 Indeed, since $\mathcal{G}_i$ are shortening functors, we have: $\ell_{ \mathcal{C}_i} (\mathcal{G}_i(A)) \leq \ell_{ \mathcal{A}} (A)$. Thus the sequence $\{ \ell_{ \mathcal{C}_i} (\mathcal{G}_i(A)) \}_i$ is bounded from above by $\ell_{ \mathcal{A}} (A)$.
 
 Now, using Corollary \ref{cor:length_obj_stab_inv_lim}, we obtain: $$\ell_{ \mathcal{C}} (\mathcal{G}(A)) = \max_{i \geq 0} \{\ell_{ \mathcal{C}_i} (\mathcal{G}_i(A)) \} \leq \ell_{ \mathcal{A}} (A)$$
 and we conclude that $\mathcal{G}$ is a shortening functor.
\end{proof}

%
%

\section{Inverse limit of categories with a filtration}\label{sec:filtr_inv_limit_def}
\subsection{}
We now consider the case when the categories $\mathcal{C}_i$ have a filtration on the objects (we will call these ``filtered categories''), and the functors $\mathcal{F}_{i-1, i}$ respect this filtration. We will then define a subcategory of the category $\varprojlim_{i \in \bZ_+} \mathcal{C}_i$ which will be denoted by $\varprojlim_{i \in \bZ_+, K-filtr} \mathcal{C}_i$ and will be called the ``inverse limit of filtered categories $\mathcal{C}_i$''.

 Fix a directed partially ordered set $(K, \leq)$ (``directed'', means that for any $k_1, k_2 \in K$, there exists $k \in K$ such that $k_1, k_2 \leq k$).

\begin{definition}[$K$-filtered categories]\label{def:filtered_categories}
 We say that a category $\mathcal{A}$ is a $K$-{\it filtered} category if for each $k\in K$ we have a full subcategory $\mathcal{A}^k$ of $\mathcal{A}$, and these subcategories satisfy the following conditions:
 
 \begin{enumerate}
  \item $\mathcal{A}^k \subset \mathcal{A}^l$ whenever $k \leq l$.
  \item $\mathcal{A}$ is the union of $\mathcal{A}^k, k\in K$: that is, for any $A \in \mathcal{A}$, there exists $k \in K$ such that $A \in \mathcal{A}^k$.
\end{enumerate}

A functor $\mathcal{F}: \mathcal{A}_1 \rightarrow \mathcal{A}_2$ between $K$-filtered categories $\mathcal{A}_1 ,\mathcal{A}_2$ is called a {\it $K$-filtered functor} if for any $k \in K$, $\mathcal{F}(\mathcal{A}^k_1)$ is a subcategory of $\mathcal{A}^k_2$.
\end{definition}
\begin{remark}
Let $\mathcal{F}: \mathcal{A}_1 \rightarrow \mathcal{A}_2$ be a $K$-filtered functor between $K$-filtered categories $\mathcal{A}_1 ,\mathcal{A}_2$. Assume the restriction of $\mathcal{F}$ to each filtration component $k$ is an equivalence of categories $\mathcal{A}^k_1 \rightarrow \mathcal{A}^k_2$. Then $\mathcal{F}$ is obviously an equivalence of ($K$-filtered) categories.
\end{remark}
\begin{remark}
 \InnaD{The definition of a $K$-filtration on the objects of a category $\mathcal{A}$ clearly makes $\mathcal{A}$ a direct limit of the subcategories $\mathcal{A}^k$.}
\end{remark}

\begin{definition}\label{def:filtr_inverse_system}
 We say that the system $((\mathcal{C}_i)_{i\in \bZ_+}, (\mathcal{F}_{i-1, i})_{i \geq 1})$ is $K$-{\it filtered} if for each $i \in \bZ_+$, $\mathcal{C}_i$ is a category \InnaD{with a $K$-filtration}, and the functors $\mathcal{F}_{i-1, i}$ are $K$-filtered functors.
 
\end{definition}

\begin{definition}\label{def:filtered_inv_limit}
Let $((\mathcal{C}_i)_{i\in \bZ_+}, (\mathcal{F}_{i-1, i})_{i \geq 1})$ be a $K$-filtered system. We define the inverse limit of this $\bZ_+$-filtered system (denoted by $\varprojlim_{i \in \bZ_+, K-filtr} \mathcal{C}_i$) to be the full subcategory of $\varprojlim_{i \in \bZ_+} \mathcal{C}_i$ whose objects $C$ satisfy: there exists $k_C \in K$ such that $\mathbf{Pr}_i(C) \in \mathcal{C}^{k_C}_i$ for any $i \in \bZ_+$.

\end{definition}

The following lemma is obvious:
\begin{lemma}
 The category $\varprojlim_{i \in \bZ_+, K-filtr} \mathcal{C}_i$ is automatically $K$-filtered: the filtration component $Fil_k(\varprojlim_{i \in \bZ_+, K-filtr} \mathcal{C}_i)$ can be defined to be the full subcategory of $\varprojlim_{i \in \bZ_+, K-filtr} \mathcal{C}_i$ of objects $C$ such that $\mathbf{Pr}_i(C) \in \mathcal{C}^{k}_i$ for any $i \in \bZ_+$. 

This also makes the functors $\mathbf{Pr}_i: \varprojlim_{i \in \bZ_+, K-filtr} \mathcal{C}_i \rightarrow \mathcal{C}_i$ $K$-filtered functors.
\end{lemma}

\begin{remark}
 Note that by definition, \InnaA{for any $k \in K$} $$Fil_k \left( \varprojlim_{i \in \bZ_+, K-filtr} \mathcal{C}_i \right) \cong \varprojlim_{i \in \bZ_+} \mathcal{C}^k_i$$ where the inverse limit is taken over the system $((\mathcal{C}^k_i)_{i \in \bZ_+}, (\mathcal{F}_{i-1, i} \rvert_{\mathcal{C}^k_i})_{i \geq 1})$. \InnaD{Thus
$$\varprojlim_{i \in \bZ_+, K-filtr} \mathcal{C}_i := \varinjlim_{k \in K} \varprojlim_{i \in \bZ_+} \mathcal{C}_i^k $$}
\end{remark}

\begin{lemma}\label{lem:filt_inv_limit_abel}
 Let $((\mathcal{C}_i)_{i\in \bZ_+}, (\mathcal{F}_{i-1, i})_{i \geq 1})$ be a $K$-filtered system.
 \begin{enumerate}
  \item  Assume the categories $\mathcal{C}_i$ are additive, the functors $\mathcal{F}_{i-1, i}$ are additive, and for any $k \in K$, $\mathcal{C}^k_i$ is an additive subcategory of $\mathcal{C}_i$.
  
  Then the category $\varprojlim_{i \in \bZ_+, K-filtr} \mathcal{C}_i$ is an additive subcategory of $\varprojlim_{i \in \bZ_+} \mathcal{C}_i$, and all its filtration components are additive subcategories.
  \item  Assume the categories $\mathcal{C}_i$ are abelian, the functors $\mathcal{F}_{i-1, i}$ are exact, and for any $k \in K$, $\mathcal{C}^k_i$ is a Serre subcategory of $\mathcal{C}_i$.
  
  Then the category $\varprojlim_{i \in \bZ_+, K-filtr} \mathcal{C}_i$ is abelian (and a Serre subcategory of $\varprojlim_{i \in \bZ_+} \mathcal{C}_i$), and all its filtration components are Serre subcategories.
 \end{enumerate}
 
\end{lemma}

\begin{proof}
 To prove the first part of the statement, we only need to check that $Fil_k(\varprojlim_{i \in \bZ_+, K-filtr} \mathcal{C}_i)$ is an additive subcategory of $\varprojlim_{i \in \bZ_+} \mathcal{C}_i$. This follows directly from the construction of direct sums in $\varprojlim_{i \in \bZ_+} \mathcal{C}_i$: let $C, D \in Fil_k(\varprojlim_{i \in \bZ_+, K-filtr} \mathcal{C}_i) \subset \varprojlim_{i \in \bZ_+} \mathcal{C}_i$. Then $\mathbf{Pr}_i(C) \in \mathcal{C}^k_i$, $\mathbf{Pr}_i(D) \in \mathcal{C}^{k}_i$ for any $i \in \bZ_+$. Since $\mathcal{C}^{k}_i$ is an additive subcategory of $\mathcal{C}_i$, we get: $\mathbf{Pr}_i (C \oplus D) \in \mathcal{C}^k_i$ for any $i \in \bZ_+$ (the direct sum $C \oplus D$ is taken in $\varprojlim_{i \in \bZ_+} \mathcal{C}_i$).
 
 Thus $\varprojlim_{i \in \bZ_+, K-filtr} \mathcal{C}_i$ is an additive subcategory of $\varprojlim_{i \in \bZ_+} \mathcal{C}_i$, and all its filtration components are additive subcategories as well.
 
 To prove the second part of the statement, it is again enough to check that $Fil_k(\varprojlim_{i \in \bZ_+, K-filtr} \mathcal{C}_i)$ is a Serre subcategory of $\varprojlim_{i \in \bZ_+} \mathcal{C}_i$.
 
 Indeed, let $$ 0 \rightarrow C' \rightarrow C \rightarrow C'' \rightarrow 0$$ be a short exact sequence in $\varprojlim_{i \in \bZ_+} \mathcal{C}_i$. We want to show that $C \in Fil_k(\varprojlim_{i \in \bZ_+, K-filtr} \mathcal{C}_i)$ iff $C', C'' \in Fil_k(\varprojlim_{i \in \bZ_+, K-filtr} \mathcal{C}_i)$.
 
 The functors $\mathbf{Pr}_i$ are exact, so the sequence $$ 0 \rightarrow C'_i \rightarrow C_i \rightarrow C''_i \rightarrow 0$$ is exact for any $i \in \bZ_+$. 
 
 Since $\mathcal{C}^k_i$ is a Serre subcategory of $\mathcal{C}_i$, we have: $C_i \in \mathcal{C}^k_i$ iff ${C'}_i, {C''}_i \in \mathcal{C}^k_i$, and we are done.
\end{proof}

We now have the following universal property, whose proof is straight-forward:
\begin{proposition}
Let $((\mathcal{C}_i)_{i \in \bZ_+}, (\mathcal{F}_{i-1, i})_{i \geq 1})$ be a $K$-filtered system, and let $\mathcal{A}$ be a \InnaD{category with a $K$-filtration}, together with a set of $K$-filtered functors $\mathcal{G}_i: \mathcal{A} \rightarrow \mathcal{C}_i$ which satisfy: for any $i \geq 1$, there exists a natural isomorphism
$$ \eta_{i-1, i}:\mathcal{F}_{i-1, i} \circ \mathcal{G}_i \rightarrow \mathcal{G}_{i-1} $$

Then $\varprojlim_{i \in \bZ_+, K-filtr} \mathcal{C}_i$ is universal among such categories; that is, we have a functor 
\begin{align*}
\mathcal{G}: \mathcal{A} &\rightarrow \varprojlim_{i \in \bZ_+, K-filtr} \mathcal{C}_i \\
A &\mapsto (\{\mathcal{G}_i(A)\}_{i \in \bZ_+}, \{\eta_{i-1, i}\}_{i \geq 1}) \\
f: A_1 \rightarrow A_2 &\mapsto \{f_i:= \mathcal{G}_i(f)\}_{i \in \bZ_+} 
\end{align*}
which is obviously $K$-filtered, and satisfies: $\mathcal{G}_i \cong \mathbf{Pr}_i \circ \mathcal{G}$ for every $i \in \bZ_+$.
\end{proposition}

Next, consider the case when $\mathcal{A}, \{\mathcal{G}_i\}_{i \in \bZ_+} $ satisfy the following ``stabilization'' condition:
\begin{cond}\label{cond:stabil_cond}
 For every $k \in K$, there exists $i_k \in \bZ_+$ such that $\mathcal{G}_j: \mathcal{A}^k \rightarrow \mathcal{C}^k_j$ is an equivalence of categories for any $j \geq i_k$.
\end{cond}

In this setting, the following proposition holds:
\begin{proposition}\label{prop:filt_inv_limit_cat_equiv}
 The functor $\mathcal{G}: \mathcal{A} \rightarrow \varprojlim_{i \in \bZ_+, K-filtr} \mathcal{C}_i$ is an equivalence of ($K$-filtered) categories.
\end{proposition}
\begin{proof}
 To prove that $\mathcal{G}$ is an equivalence of ($K$-filtered) categories, we neeed to show that 
 $$\mathcal{G}: \mathcal{A}^k \rightarrow Fil_k \left(\varprojlim_{i \in \bZ_+, K-filtr} \mathcal{C}_i \right)$$ is an equivalence of categories for any $k \in K$. Recall that $$Fil_k \left(\varprojlim_{i \in \bZ_+, K-filtr} \mathcal{C}_i \right) \cong \varprojlim_{i \in \bZ_+}  \mathcal{C}^k_i$$ By Condition \ref{cond:stabil_cond}, for any $i > i_k$ we have a commutative diagram where all arrows are equivalences:
 $$ \xymatrix{&\mathcal{A}^k  \ar[r]^-{\mathcal{G}_i} \ar[d]_{\mathcal{G}_{i_k}} & \mathcal{C}^k_i \ar[dl]^{\mathcal{F}_{i-1, i}}\\
&\mathcal{C}^k_{i-1}}
$$
 By Lemma \ref{lem:inv_limit_stabilizing_cat}, we then have:  
$\mathbf{Pr}_{i}: \varprojlim_{i \in \bZ_+}  \mathcal{C}^k_i \longrightarrow \mathcal{C}^k_{i}$ is an equivalence of categories for any $i > i_k$, and thus $\mathcal{G}: \mathcal{A}^k \rightarrow Fil_k \left(\varprojlim_{i \in \bZ_+, K-filtr} \mathcal{C}_i \right)$ is an equivalence of categories.
\end{proof}

\section{Restricted inverse limit and inverse limit of categories with a \texorpdfstring{$K$}{K}-filtration}\label{sec:stable_vs_filtr}
\subsection{}
Let $((\mathcal{C}_i)_{i\in \bZ_+}, (\mathcal{F}_{i-1, i})_{i \geq 1})$ be a system of finite-length categories with $K$-filtrations and shortening $K$-filtered functors, whose the filtration components are Serre subcategories. We would like to give a sufficient condition on the $K$-filtration for the inverse limit of $K$-filtered categories to coincide with the restricted inverse limit of these categories.

Recall that since the functors $\mathcal{F}_{i-1, i}$ are shortening, we have maps $$f_{i-1, i}: Irr_*(\mathcal{C}_i) \longrightarrow Irr_*(\mathcal{C}_{i-1})$$
and we can consider the inverse limit $\varprojlim_{i \in \bZ_+} Irr_*(\mathcal{C}_i)$ of the sequence of sets $ Irr_*(\mathcal{C}_i)$ and maps $f_{i-1, i}$; we will denote by $pr_j: \varprojlim_{i \in \bZ_+} Irr_*(\mathcal{C}_i) \rightarrow Irr_*(\mathcal{C}_j)$ the projection maps.

Notice that the sets $Irr_*(\mathcal{C}_i)$ have a natural $K$-filtration, and the maps $f_{i-1, i}$ are $K$-filtered maps. 

\begin{proposition}\label{prop:filt_vs_stab_limit}
 Assume the following conditions hold:
 \begin{enumerate}
  \item There exists a $K$-filtration on the set $\varprojlim_{i \in \bZ_+} Irr_*(\mathcal{C}_i)$. That is, we require: 

 For each $L$ in $\varprojlim_{i \in \bZ_+} Irr_*(\mathcal{C}_i)$, there exists $k \in K$ so that $pr_i(L) \in Fil_k(Irr_*(\mathcal{C}_i))$ for any $i \geq 0$.

We would then say that such an object $L$ belongs in the $k$-th filtration component of $\varprojlim_{i \in \bZ_+} Irr_*(\mathcal{C}_i)$.
  \item ``Stabilization condition'': for any $k \in K$, there exists $N_k \geq 0$ such that the map $f_{i-1, i}:Fil_k(Irr_*(\mathcal{C}_i)) \rightarrow Fil_k(Irr_*(\mathcal{C}_{i-1}))$ be an injection for any $i \geq N_k$.
  
  That is, for any $k \in K$ there exists $N_k \in \bZ_+$ such that the (exact) functor $\mathcal{F}_{i-1, i}$ is faithful for any $i \geq N_k$. 
 \end{enumerate}
Then the two full subcategories $\varprojlim_{i \in \bZ_+, \text{ restr}} \mathcal{C}_i$, $\varprojlim_{i \in \bZ_+, K-filtr} \mathcal{C}_i$ of $\varprojlim_{i \in \bZ_+} \mathcal{C}_i$ coincide.
\end{proposition}

%
%

%
\begin{proof}
Let $C: =(\{C_j\}_{j \in \bZ_+}, \{\phi_{j-1, j}\}_{j \geq 1})$ be an object in $\varprojlim_{i \in \bZ_+, \text{ restr}} \mathcal{C}_i$. As before, we denote by $JH(C_j)$ the multiset of Jordan-Holder components of $C_j$, and let 

$JH_*(C_j) := JH(C_j) \sqcup \{0\}$.

The first condition is natural: giving a $K$-filtration on the objects of $\varprojlim_{i \in \bZ_+, \text{ restr}} \mathcal{C}_i$ is equivalent to giving a $K$-filtration on the simple objects of $\varprojlim_{i \in \bZ_+, \text{ restr}} \mathcal{C}_i$, i.e. on the set $\varprojlim_{i \in \bZ_+} Irr_*(\mathcal{C}_i)$. 

Assume $C \in \varprojlim_{i \in \bZ_+, \text{ restr}} \mathcal{C}_i$. Let $n_0 \geq 0$ be such that $\ell_{\mathcal{C}_j}(C_j)$ is constant for $j \geq n_0$. 
Recall that we have (Corollary \ref{cor:param_simple_obj_stab_lim}): $$JH_*(C) = \varprojlim_{ i \in \bZ_+} JH_*(C_j)$$

Choose $k$ such that all the elements of $JH_*(C)$ lie in the $k$-th filtration component of $\varprojlim_{i \in \bZ_+} Irr_*(\mathcal{C}_i)$. This is possible due to the first condition.

Then for any $L_j \in JH(C_j)$, we have: $L_j = pr_j(L)$ for some $L \in JH_*(C)$, and thus $L_j \in Fil_k(Irr_*(\mathcal{C}_j))$. We conclude that $C \in Fil_k(\varprojlim_{i \in \bZ_+, K-filtr} \mathcal{C}_i)$.

Thus we proved that the first condition of the Theorem holds iff $\varprojlim_{i \in \bZ_+, \text{ restr}} \mathcal{C}_i$ is a full subcategory of $\varprojlim_{i \in \bZ_+, K-filtr} \mathcal{C}_i$.

\mbox{}

Now, let $C \in \varprojlim_{i \in \bZ_+, K-filtr} \mathcal{C}_i$, and let $k \in K$ be such that $C \in Fil_k(\varprojlim_{i \in \bZ_+, K-filtr} \mathcal{C}_i)$. We would like to show that $\ell_{\mathcal{C}_i}(C_i)$ is constant starting from some $i$.

Indeed, the second condition of the Theorem tells us that there exists $N_k \geq 0$ such that the map $$f_{i-1, i}: Fil_k(Irr_*(\mathcal{C}_i)) \rightarrow Fil_k(Irr_*(\mathcal{C}_{i-1}))$$ is an injection for any $i \geq N_k$. 

We claim that for $i \geq N_k$, $\ell_{\mathcal{C}_i}(C_i)$ is constant. Indeed, if it weren't, then there would be some $i \geq N_k +1 $ and some $L_i \in JH(C_i)$ such that $ f_{i-1, i}(L_i) = 0$. But this is impossible, due to the requirement above.


\end{proof}

%
%

\section{\texorpdfstring{$\gl_{\infty}$}{Infinite Lie algebra gl} and the restricted inverse limit of representations of \texorpdfstring{$\gl_n$}{finite-dimensional Lie algebras gl}}
In this section, we give a nice example of a restricted inverse limit of categories; namely, we will show that the category of polynomial representations of the Lie algebra $\gl_{\infty}$ is a restricted inverse limit of the categories of polynomial representations of $\gl_n$ for $n \geq 0$. 

The representations of the Lie algebra $\gl_{\infty}$ (or the group $GL_{\infty}$) are discussed in detail in \cite{PS}, \cite{DPS}, as well as \cite[Section 3]{SS}.

\subsection{The \texorpdfstring{Lie algebra $\gl_{\infty}$}{infinite Lie algebra gl}}\label{ssec:rep_gl_infty}
Let $\bC^{\infty}$ be a complex vector space with a countable basis $\{ e_1, e_2, e_3, ... \}$.

Consider the Lie algebra $\gl_{\infty}$ of infinite matrices $A=(a_{ij})_{i, j \geq 1}$ with finitely many non-zero entries. We have a natural action of $\gl_{\infty}$ on $\bC^{\infty}$, with $\gl_{\infty} \cong \bC^{\infty} \otimes \bC^{\infty}_{*}$. Here $\bC^{\infty}_{*} = span_{\bC}(e_1^*, e_2^*, e_3^*, ...)$, where $e_i^*$ is the linear functional dual to $e_i$: $e_i^*(e_j) = \delta_{ij}$.

\mbox{}

We now insert more notation. Let $N \in \bZ_+ \cup \{ \infty \}$, and let $m \geq 1$.
We will consider the Lie subalgebra $\gl_m \subset \gl_{N}$ consisting of matrices $A=(a_{ij})_{1 \leq i, j \leq N}$ for which $a_{ij} =0$ whenever $i>m$ or $j>m$. We will also denote by $\gl_m^{\perp}$ the Lie subalgebra of $\gl_{N}$ consisting of matrices $A=(a_{ij})_{1 \leq i, j \leq N}$ for which $a_{ij} =0$ whenever $i\leq m$ or $j \leq m$. 
\begin{remark}
 Note that $\gl_n^{\perp} \cong \gl_{N-m}$ for any $N, m$.
\end{remark}

\subsection{Categories of polynomial representations}
\mbox{}

In this subsection, $N \in \bZ_+ \cup \{\infty \}$.

We will consider the symmetric monoidal category $Rep(\gl_{N})_{poly}$ of polynomial representations of $\gl_{N}$. 

As a tensor category, it is generated by the tautological representation $\bC^N$ of $\gl_N$. Namely, this is the category of $\gl_{N}$-modules which are direct summands in finite direct sums of tensor powers of $\bC^N$, and $\gl_{N}$-equivariant morphisms between them.

This category is discussed in detail in \cite[Section 2.2]{SS}. 

It is easy to see that this is a semisimple abelian category, whose simple objects are parametrized (up to isomorphism) by all Young diagrams of arbitrary sizes: the simple object corresponding to $\lambda$ is $L^N_{\lam} = S^{\lambda} \bC^{N}$.

\begin{remark}
Note that $Rep(\gl_{\infty})_{poly}$ is the free abelian symmetric monoidal category generated by one object (c.f. \cite[(2.2.11)]{SS}). It has a equivalent definition as the category of polynomial functors of bounded degree, which can be found in \cite{HY}, \cite[Chapter I]{Mac}, \cite{SS}.
\end{remark}

\begin{remark}
 For $N \in \bZ_+$, one can describe these representations as finite-dimensional representations $\rho: GL_N \rightarrow \Aut(W)$ which can be extended to an algebraic map $\End(GL_N) \rightarrow \End(W)$.
\end{remark}

\subsection{Specialization functors}\label{ssec:spec_funct_alg}


We now define specialization functors from the category of representations of $\gl_{\infty}$ to the categories of representations of $\gl_n$ (c.f. \cite[Section 3]{SS}):
\begin{definition}\label{def:Gamma_func}
 $$\Gamma_n: Rep(\gl_{\infty})_{poly} \rightarrow Rep(\gl_{n})_{poly}, \; \Gamma_n := (\cdot)^{\gl_{n}^{\perp}}$$
\end{definition}

\begin{lemma}\label{lem:Gamma_well_def}
 The functor $\Gamma_n$ is well-defined. 
\end{lemma}

\begin{proof}
 First of all, notice that the subalgebras $\gl_{n}, \gl_{n}^{\perp} \subset \gl_{\infty}$ commute, and therefore the subspace of $\gl_{n}^{\perp}$-invariants of a $\gl_{\infty}$-module automatically carries an action of $\gl_{n}$. 
 
 We need to check that given a polynomial $\gl_{\infty}$-representation $M$ of $\gl_{n}$, the $\gl_{n}^{\perp}$-invariants of $M$ form a polynomial respresentation of $\gl_{n}$. It is enough to check that this is true when $M = (\bC^{\infty})^{\otimes r}$.
 
 The latter statement is checked explicitly on basis elements of the form $e_{i_1} \otimes e_{i_2} \otimes ... \otimes e_{i_r} $.
 The subspace of $\gl_{n}^{\perp}$-invariants is spanned by the basis elements $e_{i_1} \otimes e_{i_2} \otimes ... \otimes e_{i_r} $ for which $i_1, ..., i_r \leq  n$. Thus the $\gl_{n}^{\perp}$-invariants of $(\bC^{\infty})^{\otimes r} $ form the $\gl_{n}$-representation $(\bC^{n})^{\otimes r}$. 
 
\end{proof}
In particular, one proves in the same way that the $\gl_{n}^{\perp}$-invariants of $(\bC^{\infty})^{\otimes r} $ form the $\gl_{n}$-representation $(\bC^{n})^{\otimes r} $.
 
The following Lemmas are proved in \cite{PS}, \cite[Section 3]{SS}:
 \begin{lemma}\label{lem:Gamma_is_tensor}
 The functors $\Gamma_n$ are symmetric monoidal functors.
\end{lemma}

The functors $\Gamma_n: Rep(\gl_{\infty})_{poly} \rightarrow Rep(\gl_{n})_{poly}$ are additive functors between semisimple categories, and their effect on simple objects is given by the following Lemma (a direct consequence of Lemma \ref{lem:Gamma_is_tensor}):
  \begin{lemma}\label{lem:Gamma_simples}
 For any Young diagram $\lam$, $\Gamma_n(S^{\lam} \bC^{\infty}) \cong S^{\lam} \bC^n$.
\end{lemma}

\subsection{Restriction functors}\label{ssec:res_funct_poly}

\begin{definition}\label{def:res_funct_poly_repr}
Let $n \geq 1$. We define the functor $$\InnaD{\mathfrak{Res}}_{n-1, n}: Rep(\gl_n)_{poly} \rightarrow Rep(\gl_{n-1})_{poly}, \; \InnaD{\mathfrak{Res}}_{n-1, n} := (\cdot)^{\gl_{n-1}^{\perp}}$$
\end{definition} 
%
 The proof that this functor is well-defined is exactly the same as that of Lemma \ref{lem:Gamma_well_def}.

\begin{remark}
 Here is an alternative definition of the functors $ \InnaD{\mathfrak{Res}}_{n-1, n}$.
 
 We say that a $\gl_n$-module $M$ is of {\it degree} $d$ if $\id_{\bC^n} \in \gl_n$ acts by $d \id_M$ on $M$. Also, given any $\gl_n$-module $M$, we may consider the maximal submodule of $M$ of degree $d$, and denote it by $deg_d(M)$. This defines an endo-functor $deg_{d}$ of $Rep(\gl_n)_{poly}$.
 
 Note that a simple module $S^{\lam} \bC^n$ is of degree $\abs{\lambda}$.
 
 \mbox{}
 
 The notion of degree gives a decomposition
 $$Rep(\gl_n)_{poly} \cong \bigoplus_{d \in \bZ_+} Rep(\gl_{n})_{poly, d}$$ where $ Rep(\gl_{n})_{poly, d}$ is the full subcategory of $Rep(\gl_n)_{poly}$ consisting of all polynomial $\gl_n$-modules of degree $d$.

Then
$$ \InnaD{\mathfrak{Res}}_{n-1, n} = \oplus_{d \in \bZ_+} \InnaD{\mathfrak{Res}}_{d, n-1, n}: Rep(\gl_n)_{poly} \rightarrow Rep(\gl_{n-1})_{poly} $$ where
$$\InnaD{\mathfrak{Res}}_{d, n-1, n}: Rep(\gl_{n})_{poly, d} \rightarrow Rep(\gl_{n-1})_{poly, d} , \, \InnaD{\mathfrak{Res}}_{d, n-1, n}:= deg_{d} \circ \InnaD{\mathrm{Res}}_{\gl_{n-1}}^{\gl_n}$$
where $\InnaD{\mathrm{Res}}_{\gl_{n-1}}^{\gl_n}$ is the usual restriction functor for the pair $\gl_{n-1} \subset \gl_n$.
\end{remark}

Again, $\InnaD{\mathfrak{Res}}_{n-1, n}$ are additive functors between semisimple categories, so we are interested in checking the effect of these functors on simple modules:
\begin{lemma}\label{lem:res_func_simples}
 $\InnaD{\mathfrak{Res}}_{n-1, n}(S^{\lambda} \bC^n) \cong S^{\lambda} \bC^{n-1}$ for any Young diagram $\lam$.
\end{lemma}
\begin{proof}
 This is a simple corollary of the branching rues for $\gl_n, \gl_{n-1}$. 
\end{proof}

Next, we notice that these functors are compatible with the functors $\Gamma_n$ defined before. 
\begin{lemma}\label{lem:Gamma_res_compat}
 For any $n \geq 1$, we have a commutative diagram:
 $$\xymatrix{&Rep(\gl_{\infty})_{poly}  \ar[r]^{\Gamma_n} \ar[rd]_{\Gamma_{n-1}} &Rep(\gl_n)_{poly} \ar[d]^{\InnaD{\mathfrak{Res}}_{n-1, n}} \\ &{} &Rep(\gl_{n-1})_{poly} }$$
 That is, there is a natural isomorphism $\Gamma_{n-1}  \cong \InnaD{\mathfrak{Res}}_{n-1, n} \circ \Gamma_n $.
\end{lemma}
\begin{proof}
By definition of the functors $\Gamma_{n-1},\InnaD{\mathfrak{Res}}_{n-1, n},\Gamma_n $, we have a natural transformation $\theta: \Gamma_{n-1}  \rightarrow \InnaD{\mathfrak{Res}}_{n-1, n} \circ \Gamma_n $ which is given by the injection $\theta_M: \Gamma_{n-1}(M) \hookrightarrow \left( \InnaD{\mathfrak{Res}}_{n-1, n} \circ \Gamma_n \right) (M) $ for any $M \in Rep(\gl_{\infty})_{poly}$. We would like to say that $\theta_M$ are isomorphisms.

 The categories in question are semisimple, so it is enough to check what happens to the simple objects. Lemmas \ref{lem:Gamma_simples} and \ref{lem:res_func_simples} then tell us that $\theta_{S^{\lam} \bC^{\infty}}$ is an isomorphism for any Young diagram $\lam$, and we are done.
\end{proof}

\begin{lemma}\label{lem:res_func_tensor}
 The functors $\InnaD{\mathfrak{Res}}_{n-1, n}:Rep(\gl_{n})_{poly} \rightarrow Rep(\gl_{n-1})_{poly}$ are symmetric monoidal functors.
\end{lemma}
\begin{proof}
The functor $\Gamma_n$ is full and essentially surjective, as well as a tensor functor. The natural isomorphism from Lemma \ref{lem:Gamma_res_compat} then provides a monoidal structure on the functor $\InnaD{\mathfrak{Res}}_{n-1, n}$, and we can immediately see that it is symmetric as well.
\end{proof}

\subsection{The restricted inverse limit of categories \texorpdfstring{$Rep(\gl_n)_{poly}$}{of polynomial representations}}\label{ssec:Stab_inv_lim_rep_poly}
This subsection describes the category $Rep(\gl_{\infty})_{poly}$ as a ``stable'' inverse limit of categories $Rep(\gl_n)_{poly}$. 

We now define a $\bZ_+$-filtration on $Rep(\gl_n)_{poly}$ for each $n \in \bZ_+$.
\begin{notation}
 For each $k \in \bZ_+$, let $Rep(\gl_n)_{poly, \text{ length } \leq k}$ be the full additive subcategory of $Rep(\gl_n)_{poly}$ generated by $S^{\lambda} \bC^n$ such that $\ell(\lambda) \leq k$. 
\end{notation}
Clearly the subcategories $Rep(\gl_n)_{poly, \text{ length } \leq k}$ give us a $\bZ_+$-filtration of the category $Rep(\gl_n)_{poly}$, and by Lemma \ref{lem:res_func_simples}, the functors $ \InnaD{\mathfrak{Res}}_{n-1, n}$ are $\bZ_+$-filtered functors (see Section \ref{sec:filtr_inv_limit_def}).

This allows us to consider the inverse limit $$\varprojlim_{n \in \bZ_+, \bZ_+-filtr} Rep(\gl_n)_{poly}$$ of $\bZ_+$-filtered categories $Rep(\gl_n)_{poly}$. This inverse limit is an abelian category \InnaD{with a $\bZ_+$-filtration} (by Lemma \ref{lem:filt_inv_limit_abel}).

Note that by Lemma \ref{lem:res_func_simples}, the functors $\InnaD{\mathfrak{Res}}_{n-1, n}$ are shortening functors (see Definition \ref{def:shorten_functor}); futhermore, the system $((Rep(\gl_n)_{poly})_{n \in \bZ_+}, (\InnaD{\mathfrak{Res}}_{n-1, n})_{n \geq 1})$ satisfies the conditions in Proposition \ref{prop:filt_vs_stab_limit}, and therefore the inverse limit of this $\bZ_+$-filtered system is also its restricted inverse limit (see Section \ref{sec:stab_inv_lim}).

Of course, since the functors $\InnaD{\mathfrak{Res}}_{n-1, n}$ are symmetric monoidal functors, the above restricted inverse limit is a symmetric monoidal category.

\begin{proposition}\label{prop:inv_lim_cat_poly_rep}
 We have an equivalence of symmetric monoidal abelian categories 
 $$\Gamma_{\text{lim}}: Rep(\gl_{\infty})_{poly} \longrightarrow \varprojlim_{n \in \bZ_+, \text{ restr}}  Rep(\gl_{n})_{poly} $$
 induced by the symmetric monoidal functors 
 $$\Gamma_n =  ( \cdot )^{\gl_n^{\perp}}:  Rep(\gl_{\infty})_{poly} \longrightarrow  Rep(\gl_{n})_{poly}$$
\end{proposition}
 \begin{proof}
Define a $\bZ_+$-filtration on the semisimple category $ Rep(\gl_{\infty})_{poly}$ by requiring the simple object $S^{\lam} \bC^{\infty}$ to lie in filtra $\ell(\lam)$. Lemma \ref{lem:Gamma_simples} then tells us that for any $k \in \bZ_+$ and any $n \geq k$, the functor 
$$\Gamma_n: Fil_k(Rep(\gl_{\infty})_{poly}) \longrightarrow Fil_k(Rep(\gl_{n})_{poly}) := Rep(\gl_n)_{poly, \text{ length } \leq k}$$ is an equivalence. Proposition \ref{prop:filt_inv_limit_cat_equiv} completes the proof.
 \end{proof}

 \begin{remark}
 The same result has been proved in \cite{HY}; the approach used there is equivalent to that of inverse limits of $\bZ_+$-filtered categories - namely, the authors give a $\bZ_+$-grading on the objects of each category $Rep(\gl_{n})_{poly}$, with $S^{\lambda} \bC^n$ lying in grade $\abs{\lambda}$. The ``stable'' inverse limit of these graded categories, as defined in \cite{HY}, is just the inverse limit of the $\bZ_+$-filtered categories $Rep(\gl_{n})_{poly}$ with the appropriate filtrations. Note that by Proposition \ref{prop:filt_vs_stab_limit}, this construction is equivalent to our construction of a $\varprojlim_{n \in \bZ_+, \text{ restr}}  Rep(\gl_{n})_{poly} $.

 In this case, this is also equivalent to taking the compact subobjects inside $\varprojlim_{n \in \bZ_+}  Rep(\gl_{n})_{poly} $.
\end{remark}
\begin{remark}
  The adjoint (on both sides) to functor $\Gamma_{\text{lim}}$ is the functor 
 $$\Gamma_{\text{lim}}^*:\varprojlim_{n \in \bZ_+, \text{ restr}}  Rep(\gl_{n})_{poly} \longrightarrow  Rep(\gl_{\infty})_{poly} $$ defined below.
 
 For any object $((M_n)_{n\geq 0}, (\phi_{n-1, n})_{n \geq 1})$ of $\varprojlim_{n \in \bZ_+, \text{ restr}}  Rep(\gl_{n})_{poly}$, the $\gl_{n-1}$-module $M_{n-1}$ is isomorphic (via $\phi_{n-1, n}$) to a $\gl_{n-1}$-submodule of $M_n$. 
 
 This allows us to cosider a vector space $M$ which is the direct limit of the vector spaces $M_n$ and the inclusions $\phi_{n-1, n}$. On this vector space $M$ we have a natural action of $\gl_{\infty}$: given $A \in \gl_n \subset \gl_{\infty}$ and $m \in M$, we have $m \in M_N$ for $N>>0$. In particular, we can choose $N \geq n$, and then $A$ acts on $m$ through its action on $M_N$.
 
  We can easily check that the $\gl_{\infty}$-module $M$ is polynomial: indeed, due to the equivalence in Proposition \ref{prop:inv_lim_cat_poly_rep}, there exists a polynomial $\gl_{\infty}$-module $M'$ such that $M_n \cong \Gamma_n(M')$ for every $n$, and $\phi_{n-1, n}$ are induced by the inclusions $\Gamma_{n-1}(M') \subset \Gamma_{n}(M')$. By definition of $M$, we have a $\gl_{\infty}$- equivariant map $M \rightarrow M'$, and it is easy to check that it is an isomorphism.
  
 We put $\Gamma_{\text{lim}}^*((M_n)_{n\geq 0}, (\phi_{n-1, n})_{n \geq 1}):= M$, and require that the functor $\Gamma_{\text{lim}}^*$ act on morphisms accordingly. The above construction then gives us a natural isomorphism $$ \Gamma_{\text{lim}}^* \circ \Gamma_{\text{lim}} \stackrel{\sim}{\longrightarrow} \id_{Rep(\gl_{\infty})_{poly}}. $$
 
\end{remark}


\end{document}